\documentclass[9pt]{amsart}
\usepackage{amsmath,amssymb,amsthm,mathtools,dsfont,bm}

\usepackage[dvips]{graphicx}
\usepackage{color}
\usepackage{wasysym}
\newtheorem{lem}{Lemma}[section]
\newtheorem{prop}{Proposition}[section]

\newtheorem{thm}{Theorem}[section]

\theoremstyle{definition}

\theoremstyle{remark}

\theoremstyle{remark}
\newtheorem{remark}{Remark}[section]
\newtheorem{assum}{Assumption}
\numberwithin{equation}{section}

\newcommand{\C}{{\mathbb C}}
\newcommand{\N}{{\mathbb N}}

\newcommand{\R}{{\mathbb R}}

\definecolor{blu}{rgb}{0,0,1}

\newcommand{\pt}{\partial}
\newcommand{\DD}{\Delta}
\newcommand{\hD}{\sqrt{-\Delta}}

\newcommand{\Rcal}{\mathcal{R}}

\newcommand{\be}{\begin{equation}}
\newcommand{\ee}{\end{equation}}

\newcommand{\eps}{\varepsilon}

\newcommand{\Copt}{\mathbf{C}}

\newcommand{\weakto}{\rightharpoonup}

\catcode`@=11
\def\section{\@startsection{section}{1}%
  \z@{1.5\linespacing\@plus\linespacing}{.5\linespacing}%
  {\normalfont\bfseries\large\centering}}
\catcode`@=12

\begin{document}
\title[On Traveling Solitary Waves for Half-Wave Equations]{On Traveling Solitary Waves and  Absence of Small Data Scattering for  Nonlinear Half-Wave Equations}
\author{Jacopo Bellazzini}
\address{J. Bellazzini,
\newline  Universit\`a di Sassari, Via Piandanna 4, 70100 Sassari, Italy}%
\email{jbellazzini@uniss.it}%
\author{Vladimir Georgiev}
\address{V. Georgiev
\newline Dipartimento di Matematica Universit\`a di Pisa
Largo B. Pontecorvo 5, 56100 Pisa, Italy\\
 and \\
 Faculty of Science and Engineering \\ Waseda University \\
 3-4-1, Okubo, Shinjuku-ku, Tokyo 169-8555 \\
Japan and IMI--BAS, Acad.
Georgi Bonchev Str., Block 8, 1113 Sofia, Bulgaria}%
\email{georgiev@dm.unipi.it}%
\author{Enno Lenzmann}
\address{E. Lenzmann,
\newline Universit\"at Basel, Departement Mathematik und Informatik, Spiegelgasse 1, CH-4051 Basel, Switzerland}%
\email{enno.lenzmann@unibas.ch}
\author{Nicola Visciglia}
\address{N. Visciglia, \newline Dipartimento di Matematica Universit\`a di Pisa
Largo B. Pontecorvo 5, 56100 Pisa, Italy}%
\email{viscigli@dm.unipi.it}
\begin{abstract}
We consider nonlinear half-wave equations with focusing power-type nonlinearity
$$
i \pt_t u = \sqrt{-\Delta} \, u - |u|^{p-1} u, \quad \mbox{with $(t,x) \in \R \times \R^d$}
$$
with exponents $1 < p < \infty$ for $d=1$ and $1 < p < (d+1)/(d-1)$ for $d \geq 2$. We study traveling solitary waves of the form
$$
u(t,x) = e^{i\omega t} Q_v(x-vt)
$$
with frequency $\omega \in \R$, velocity $v \in \R^d$, and some finite-energy profile $Q_v \in H^{1/2}(\R^d)$, $Q_v \not \equiv 0$. We prove that traveling solitary waves for speeds $|v| \geq 1$ do not exist. Furthermore, we generalize the non-existence result to the square root Klein--Gordon operator $\sqrt{-\DD+m^2}$ and other nonlinearities.

As a second main result, we show that small data scattering fails to hold for the focusing half-wave equation in any space dimension. The proof is based on the existence and properties of traveling solitary waves for speeds $|v| < 1$. Finally, we discuss the energy-critical case when $p=(d+1)/(d-1)$ in dimensions $d \geq 2$.
\end{abstract}

\maketitle

\section{Introduction and Main Results}

\label{sec:intro}

In this paper, we consider half-wave equations in arbitrary space dimensions with focusing power-type nonlinearity. The corresponding Cauchy problem with initial data in the Sobolev space $H^{1/2}(\R^d)$ is given by
\be \tag{HW} \label{eq:HW}
\left \{ \begin{array}{l} i \pt_t u = \hD  u - |u|^{p-1} u , \\
u(0,x) = u_0(x) , \quad u_0 \in H^{1/2}(\R^d) . \end{array} \right .
\ee
Here $p > 1$ is a real exponent that satisfies a suitable condition stated in \eqref{eq:p_cond} below. As usual, the operator $\hD$ denotes the square root of Laplacian in $\R^d$ defined in Fourier space through  $\hD u = \mathcal{F}^{-1} ( |\xi| \mathcal{F} )$. Let us mention that nonlinear half-wave equations of the form \eqref{eq:HW} or similar types have recently attracted some substantial attention in the area of dispersive nonlinear PDE. From the physical point of view, the motivation to study such as equations ranges from turbulence phenomena, wave propagation, continuum limits of long-range lattice system, and models for gravitational collapse in astrophysics; see, e.\,g., \cite{ES,KLS,FL2,We2}.

On a formal level, the Cauchy problem \eqref{eq:HW} enjoys the conservation of energy and $L^2$-mass, which are given by the expressions
\be
E(u) = \frac 1 2 \int_{\R^d} \overline{u} \hD u \, dx - \frac{1}{p+1} \int_{\R^d} |u|^{p+1} \, dx, \quad M(u) = \int_{\R^d} |u|^2 \, dx,
\ee
respectively. As a consequence, the space $H^{1/2}(\R^d)$ arises as a natural choice for the energy space related to (HW). Mostly in what follows, we  assume that the real exponent $p > 1$ in (HW) is energy-subcritical, which means that
\be \label{eq:p_cond}
1 < p < p_*  \quad \mbox{with} \quad p_* := \begin{dcases*} \infty & for $d=1$, \\ \displaystyle \frac{d+1}{d-1} & for $d \geq 2$. \end{dcases*}
\ee
Below, we shall also make some comments on the energy-critical case $p=p_*$ in dimensions $d \geq 2$. For well-posedness and blowup results results about the Cauchy problem (HW), we refer the reader to \cite{KLR,BGV,OV}.

It is fair to say that the nonlinear half-wave equation shares many features with classical nonlinear Schr\"odinger equations (NLS) and nonlinear wave equations (NLW). In particular, due to the focusing nature of the nonlinearity, we have existence of {\em solitary waves} for (HW). More precisely, for a given velocity $v \in \R^d$ and frequency parameter $\omega \in \R$, we can seek {\em traveling solitary waves} for \eqref{eq:HW} with finite energy. By definition, these are solutions of the form
\be
u(t,x) = e^{it \omega} Q_{v,\omega}(x-v t) \in H^{1/2}(\R^d).
\ee
A crucial fact, however, is that (HW) does not enjoy any known Galilean or Lorentz boost symmetry, which would enable us to explicitly obtain traveling solitary waves with profile $Q_v \in H^{1/2}(\R^d)$ from profiles at rest corresponding to the case $v=0$. To compensate for the lack of such an explicit boost transformation, we study the equation satisfied by $Q_v$ itself, i.\,e.,
\be \label{eq:Qv}
\hD Q_{v,\omega} + i v \cdot \nabla Q_{v,\omega} + \omega Q_{v,\omega} - |Q_{v,\omega}|^{p-1} Q_{v,\omega} = 0.
\ee
In view of its symbol in Fourier space, the properties of the pseudo-differential operator $\hD + i v \cdot \nabla$ will crucially depend on the size of the traveling speed, i.\,e.~whether it holds that
$$
|v| < 1 \quad \mbox{or} \quad |v| \geq 1.
$$
As the main results in the present paper, we will establish the following two results about (HM).
\begin{itemize}
\item Traveling solitary waves with speeds $|v| \geq 1$ do no exist.
\item Small data in energy space do not scatter to free solutions in general.
\end{itemize}
In fact, the absence of small data will be due to the fact that traveling solitary with arbitrarily small energy are shown to exist for speeds $|v| < 1$ close to 1. We mention that the general failure of small data scattering for (HM) is in striking contrast to the well-known results for (NLS) and (NLW).

\subsection{Non-Existence of Traveling Solitary Waves for $|v| \geq 1$}
As our first main result, we rule out the existence of (non-trivial) traveling solitary waves for (HW) with speeds larger the unity (in our choice of physical units here).

\begin{thm} \label{thm:nofast}
Let $d \in \N$ and $1 < p < p_*$. Suppose $Q_{v,\omega} \in H^{1/2}(\R^d)$ solves \eqref{eq:Qv} with some $v \in \R^d$ and $\omega \in \R$. Then  it holds
$$
Q_{v,\omega} \equiv 0,
$$
provided that one of the following conditions holds:
\begin{enumerate}
 \item[(i)] $|v| > 1$,
 \item[(ii)] $|v|=1$ and $d=1$.
 \end{enumerate}
\end{thm}

\begin{remark}
In Section \ref{sec:non_exist} below, we will prove a more general result for equations of the form
$$
\sqrt{-\DD+m^2} \, Q_v + i v \cdot \nabla Q_{v} + f(|Q_v|) Q_v= 0,
$$
where $f : [0,\infty) \to \R$ is some general nonlinearity and $\sqrt{-\DD+m^2}$ with $m \geq 0$ is the square root Klein-Gordon operator. Furthermore, the arguments laid out in Section \ref{sec:non_exist} will also carry over to nonlocal focusing nonlinearities. For instance, it follows that traveling solitary waves for {\em boson star equation} (see \cite{ES, FL2})
\be \label{eq:boson}
i \pt_t u = \sqrt{-\Delta+m^2} \, u - (|x|^{-1} \ast |u|^2) u, \quad \mbox{with $(t,x) \in \R \times \R^3$}
\ee
do not exist with speeds $|v| > 1$.
\end{remark}

\begin{remark}
We believe that Theorem \ref{thm:nofast} can be extended to $|v|=1$ in higher dimensions $d \geq 2$. The interested reader may consult Section \ref{sec:non_exist} to learn where the condition $d=1$ enters when treating the limiting case $|v|=1$.
\end{remark}

\begin{remark}
For speeds $|v| < 1$, non-trivial solutions $Q_v \in H^{1/2}(\R^d)$ for equation \eqref{eq:Qv} do exist. This fact will be exploited in the next subsection when we show a general non-scattering result for \eqref{eq:HW} in energy space.
\end{remark}

Let us comment on the strategy behind the proof of Theorem \ref{thm:nofast}. A formal starting point for the analysis is to integrate \eqref{eq:Qv} against $x \overline{Q}_v$ and to take imaginary parts. Indeed, if we proceed in this formal fashion and pretend that integration by parts can be justified, we arrive at the following set of identities
\be \label{eq:virial_intro}
\int_{\R^d} \overline{Q}_v (i \mathcal{R}_k Q_v) \, dx = v_k \int_{\R^d} |Q_v|^2 \, dx
\ee
for all $k=1, \ldots, d$. Here the operators $\mathcal{R}_k$ denote the classical {\em Riesz transforms} in $L^2(\R^d)$ defined in Fourier space by
$$
\widehat{(\mathcal{R}_k f)}(\xi) = -i \frac{\xi_k}{|\xi|} \widehat{f}(\xi).
$$
For $d=1$ space dimension, we recall that $\mathcal{R}_1=H$ is the {\em Hilbert transform} on the real line. If we take \eqref{eq:virial_intro} for granted, the obvious $L^2$-bound $\| \mathcal{R}_k f \|_{L^2} \leq \| f \|_{L^2}$ applied to \eqref{eq:virial_intro} then shows that $|v| > 1$ implies that $\|Q_v\|_{L^2}=0$ must hold. This concludes our formal proof of Theorem \ref{thm:nofast} part (i). [The case (ii) needs a refined argument also to show that $Q_v=0$ has to vanish identically.]

Of course, the formal argument sketched above calls for a rigorous justification. However, an approach that initially tries to establish some needed spatial decay by proving that $x Q_v \in L^2(\R^d)$ will be plagued with some serious difficulties due to the assumption $|v| > 1$ (with some delicate effort one can deal with $|v|=1$). In fact, for speeds $|v| > 1$, it seems unclear how to prove some spatial decay of $Q_v$ because the differential operator $\hD + i v \cdot \nabla+\omega$ does not have a resolvent with good decay bounds in $x$-space. To overcome this obstacle, we shall follow a different path by avoiding of expressions such as $x Q_v$ at all. Instead of using the unbounded (mutliplication) operator $f \mapsto xf$, we consider the family of unitary operators on $L^2(\R^d)$ generated by $x$. That is, for $a \in \R^d$ given, we set
\be
(U_a f)(x) = e^{i a \cdot x} f(x)
\ee
defined for $f \in L^2(\R^d)$. With the use of the unitary family $U_a$ and by taking carefully the limit $a \to 0$, we can devise a well-defined argument that will lead us to \eqref{eq:virial_intro}. A further benefit from using $U_a$ is that our proof easily generalizes to the square root Klein--Gordon operator $\sqrt{-\Delta^2+m^2}$; see Section \ref{sec:non_exist} below.

Finally, as an interesting aside, we mention that this scheme of using unitary families is reminiscent to the classical proof of the virial theorem for Schr\"odinger operators due to Weidmann \cite{Weid}, where instead of the unbounded self-adjoint operator of dilations $A= -\frac{i}{2} (x \cdot \nabla + \nabla \cdot x)$ in $L^2(\R^d)$ its corresponding unitary one-parameter group $\{ e^{i \eps A} \}_{\eps \in \R}$ is used.

\subsection{Failure of Small Data Scattering}

Our second main result concerns the failure of small data scattering for (HW) in energy space. We have the following general result.

\begin{thm} \label{thm:no_scat}
Let $d \in \N$ and $1 < p < p_*$. Then the initial-value problem \eqref{eq:HW} has no small data scattering in energy space.

More precisely, for any $\eps > 0$, there is a global-in-time solution $u \in C^0([0, \infty); H^{1/2}(\R^d))$ of \eqref{eq:HW} with $0 < \| u(t) \|_{H^{1/2}} < \eps$ and the property such that {\bf no} function $\phi \in H^{1/2}(\R^d)$ exists with
$$
\lim_{t \to +\infty} \| u(t) - e^{-it \hD} \phi \|_{H^{1/2}} = 0.
$$
An analogous statement is true for negative times $t \to -\infty$.
\end{thm}

\begin{remark}
The non-scattering solutions will be given by small traveling solitary waves whose profile satisfies $\| Q_v \|_{H^{1/2}} < \eps$, which will be achieved by choosing speeds $|v| < 1$ close to 1. We note that such $Q_v(x)$ and hence $u(t,x)$ are {\em not radially symmetric} in $x \in \R^d$. Thus it may well be true that small data scattering for (HW) in energy space does hold if one restricts to {\em radial data}. Any proof (or disproof) of this claim would be very interesting.
\end{remark}

\begin{remark}
The non-scattering solutions provided by Theorem \ref{thm:no_scat} can also be used to show small data scattering for (HW) also fails in the corresponding critical Sobolev norm $\| \cdot \|_{\dot{H}^{s_c}}$ where
$$
s_c = \frac{d}{2} - \frac{1}{p-1}.
$$
More precisely, if $s_c \geq 0$, then small initial data for (HW) in $\dot{H}^{s_c}(\R^d)$ cannot scatter to a free solution in general. In particular, this is in drastic contrast to the celebrated result for $L^2$-critical focusing NLS
\end{remark}

\begin{remark}
For the $L^2$-critical focusing half-wave equation in one-dimension (i.\,e.~$d=1$ and $p=3$), the existence of non-scattering solutions with arbitrarily small $L^2$-mass was observed in \cite{KLR}. However, we find remarkable that such phenomenon also carries over to higher dimensions $d \geq 2$, where the linear half-wave propagator $e^{-it \hD}$ has dispersive properties.
\end{remark}

Let us briefly comment on the proof of Theorem \ref{thm:no_scat}. The idea is to construct (non-trivial) traveling solitary waves for (HW) with velocity $v \in \R^d$ such that $|v| < 1$ by considering the variational problem
\be  \label{def:sup}
\sup_{u \in H^{1/2}(\R^d) \setminus \{ 0 \} }  \frac{\displaystyle \int_{\R^d} |u|^{p+1} \, dx}{(T_v(u))^{\frac{d(p-1)}{2}} M(u)^{\frac{p+1}{2} - \frac{d(p-1)}{2}}},
\ee
where  we recall that $M(u) = \int_{\R^d} |u|^2 \, dx$ denotes the $L^2$-mass and we introduce the functional
$$
T_v(u) = \int_{\R^d} \overline{u} ( \hD + i v \cdot \nabla ) u \, dx.
$$
By adapting known methods from the calculus of variations, one finds that the supremum in \eqref{def:sup}  is indeed attained and that its optimizers $Q_v \in H^{1/2}(\R^d)$ solve the profile \eqref{eq:Qv}, after a possible rescaling. The key for proving Theorem \ref{thm:no_scat} is now to show that we can find optimizers such that
$$
\|Q_v \|_{H^{1/2}} \to 0 \quad \mbox{as} \quad  |v| \to 1^-.
$$
See Section \ref{sec:no_scatt} for more details.

Finally, we mention that, for the special case $p=3$ and $d=1$, such a variational construction of small traveling solitary waves was done in \cite{KLR} for the first time. Later is was used to construct turbulent two-soliton solutions in \cite{GLPR} for the cubic half-wave equation in one dimension, by exploiting a certain regime close to the completely integrable cubic Szeg\"o equation.

\subsection{On the Energy-Critical Case}
Let us explain  how the results above extend to the half-wave equation in the energy-critical case when $p=p_*$ and $d \geq 2$. That is, we consider the equation
\be \label{eq:HW_crit}
i \pt_t u = \hD \, u - |u|^{\frac{2}{d-1}} u, \quad (t,x) \in \R \times \R^d.
\ee
The natural choice for initial data is now the homogeneous Sobolev space $\dot{H}^{1/2}(\R^d)$. This fact poses some technical difficulties, since the use of an identity like \eqref{eq:virial_intro} presuppose that the traveling solitary wave profiles $Q_v$ also belong to $L^2(\R^d)$. Hence ruling out traveling solitary waves for \eqref{eq:HW_crit} in $H^{1/2}(\R^d)$ with speeds $|v| > 1$ follows from a straightforward adaption of the argument below. However, the more general case for profiles $Q_v \in \dot{H}^{1/2}(\R^d)$ is left as an interesting open problem.

On the other hand, the failure of small data scattering for energy-critical half-wave equation \eqref{eq:HW_crit} can carried out along the same lines as in the energy-subcritical case. We have the following result.

\begin{thm} \label{thm:HW_crit}
Let $d \geq 2$. Then the energy-critical half-wave equation \eqref{eq:HW_crit} has no small data scattering in energy space.

More precisely, for any $\eps > 0$, there is a global-in-time solution $u \in C^0([0, \infty); \dot{H}^{1/2}(\R^d))$ of \eqref{eq:HW_crit} with $0 < \| u(t) \|_{\dot{H}^{1/2}} < \eps$ and the property such that {\bf no} function $\phi \in  \dot{H}^{1/2}(\R^d)$ exists with
$$
\lim_{t \to +\infty} \| u(t) - e^{-it \hD} \phi \|_{\dot{H}^{1/2}} = 0.
$$
An analogous statement is true for negative times $t \to -\infty$.
\end{thm}

The proof of Theorem \ref{thm:HW_crit} is given in the Appendix of this paper.

\begin{remark}
The failure of small data scattering for the energy-critical half-wave equation \eqref{eq:HW_crit} is again in striking contrast to the behavior for (NLS) and (NLW) with focusing energy-critical nonlinearity; see \cite{KM,KM2,D}.
\end{remark}

\subsection*{Acknowledgments}
J. \,B. and V.\,G. are partially  supported  by   Project 2017 "Problemi stazionari e di evoluzione nelle equazioni di campo nonlineari" of INDAM, GNAMPA - Gruppo Nazionale per l'Analisi Matematica, la Probabilita e le loro Applicazioni. V.\,G. is also  partially supported by Institute of Mathematics and Informatics, Bulgarian Academy of Sciences, by Top Global University Project, Waseda University and  the Project PRA 2018 49 of  University of Pisa. E.\,L.~was partially supported by the Swiss National Science Foundation (SNSF) through Grant No.~200021--149233. In addition, E.~L.~thanks Rupert Frank for valuable discussions and providing us with reference \cite{Brezis}.

\section{Non-Existence of Traveling Solitary Waves for $|v| \geq 1$}

\label{sec:non_exist}
This section is devoted to the proof of Theorem \ref{thm:nofast}. In fact, we shall establish a more general non-existence result for traveling solitary waves with speeds $|v| \geq 1$; see Propositions \ref{prop:no_Qv_1} and \ref{prop:no_Qv_2} below. Throughout this section, we use
$$
(u,v) = \int_{\R^d} \overline{u(x)} v(x) \, dx
$$
to denote the standard complex scalar product in $L^2(\R^d)$.

We first derive the following key result, which can be seen as a {\bf virial-type identity}, which would {\em formally} follow by testing the equation against $x \overline{Q}_v$. For more on this formal point of view, we refer to the Section \ref{sec:intro} above.
\begin{lem} \label{lem:virial}
Let $d \in \N$, $m \geq 0$, and $v =(v_1, \ldots, v_d) \in \R^d$. Suppose $Q_v \in H^{1/2}(\R^d)$ solves
$$
\sqrt{-\DD+m^2} \, Q_v + i v\cdot \nabla Q_v + f(|Q_v|) Q_v = 0 \quad \mbox{in $H^{-1/2}(\R^d)$,}
$$
where $f : [0,\infty) \to \R$ is some real-valued function. Then the identity
$$
\int_{\R^d} \frac{\xi_k}{\sqrt{|\xi|^2 + m^2}} |\widehat{Q}_v(\xi)|^2 \, d\xi = v_k \int_{\R^d} |\widehat{Q}_v(\xi)|^2 \, d \xi
$$
holds for all $k=1, \ldots, d$.
\end{lem}

\begin{proof}
For $a \in \R^d$, we define the unitary family $U_a$ acting on $L^2(\R^d)$ by setting
\be
(U_a \phi)(x) = e^{i a \cdot x} \phi(x).
\ee
It is elementary to check that $U_a \phi \in H^{1/2}(\R^d)$ when $\phi \in H^{1/2}(\R^d)$. For notational convenience, we denote
$$
Q = Q_v, \quad D_v = \sqrt{-\DD + m^2} + i v \cdot \nabla, \quad \mbox{and $V(x) = f(|Q_v(x)|)$}.
$$
Thus the equation satisfied by $Q \in H^{1/2}(\R^d)$ can be written as
\be \label{eq:virial1}
D_v Q + V Q = 0.
\ee
For any $a \in \R^d$, the function $Q_{a} := U_a Q \in H^{1/2}(\R^d)$ is easily seen to satisfy the equation
\be \label{eq:virial2}
D_{v}^a Q_{a} + V Q_{a} = 0,
\ee
where we use that $V = U_{a} V U_a^{-1}$ (since $V$ commutes with $U_a$) and we denote
$$
D_{v}^a := U_a D_v U_a^{-1} = \mathcal{T}_a + (a \cdot v) \mathds{1} +iv \cdot \nabla.
$$
Here $\mathcal{T}_a$ is the operator given in Fourier space by
$$
\widehat{(\mathcal{T}_a f)}(\xi) = \sqrt{|\xi-a|^2 + m^2} \, \widehat{f}(\xi).
$$

Next, we use \eqref{eq:virial1} and \eqref{eq:virial2} together with the fact that $V$ is real-valued to conclude
$$
(Q_a, (D_{v}^a- D_v) Q) = (D^a_v Q_a, Q) - (Q_a, D_v Q) =-( V Q_a, Q) + ( Q_a, V Q) = 0.
$$
That is, we have
\be \label{eq:virial3}
(Q_a, (D_{v}^a- D_v) Q) = 0 \quad \mbox{for any $a \in \R$}.
\ee

Next, we fix some $k \in \{1, \ldots, d\}$ and consider $a= h e_k$ with $h \neq 0$, where  $e_k \in \R^d$ denotes the $k$-th unit vector. We claim that
\be \label{eq:virial4}
\mbox{$\frac{1}{h} ( D_v^{h e_k} - D_v) Q \rightharpoonup -i\Rcal_k Q  + v_k Q$ weakly in $L^2(\R)$ as $h \to 0$},
\ee
where $\Rcal_k$ denotes the bounded operator on $L^2(\R^d)$ given by
$$
\widehat{(\Rcal_k f)}(\xi) = \frac{-i \xi_k}{\sqrt{|\xi|^2 + m^2}} \widehat{f}(\xi).
$$
For $m=0$, we recall the operators $\{ \Rcal_k\}_{k=1}^d$ are the usual Riesz transforms in $\R^d$. To show \eqref{eq:virial4}, we observe that, for any $\phi \in L^2(\R)$ fixed,
\begin{align*}
(\phi, h^{-1} (D_v^{h e_k}-D_v) Q) & =  \int_{\R^d} \overline{\widehat{\phi}(\xi)} \left ( \frac{1}{h} \left \{ \sqrt{|\xi-h e_k|^2 + m^2} - \sqrt{|\xi|^2+m^2} \right \} +  v_k \right ) \widehat{Q}(\xi) \, d \xi  \\
& \to \int_{\R^d}  \overline{\widehat{\phi}(\xi)} \left ( \frac{-\xi_k}{\sqrt{|\xi|^2 + m^2}} + v_k \right ) \widehat{Q}(\xi) \, d \xi \quad \mbox{as} \quad h \to 0,
\end{align*}
by dominated convergence using  that $|h|^{-1} | \sqrt{|\xi - h e_k|^2 +m^2} - \sqrt{|\xi|^2+m^2} | \leq 1$ for all $\xi \in \R$ and $h \neq 0$. Thus we have shown that \eqref{eq:virial4} holds.

Finally, since $Q_a \to Q$ strongly in $L^2(\R^d)$ as $a \to 0$, we conclude from \eqref{eq:virial3} and \eqref{eq:virial4} that
\be
(Q, i \Rcal_k Q) = v_k (Q,Q)  \quad \mbox{for $k=1, \ldots, d$},
\ee
which is the claimed identity due to Plancherel's identity.
\end{proof}

\begin{prop} \label{prop:no_Qv_1}
Let $Q_v \in H^{1/2}(\R^d)$ be as in Lemma \ref{lem:virial} above. If we have
$$
\mathrm{(i)} \ \ |v| > 1 \quad \mbox{or} \quad \mathrm{(ii)} \ \ \mbox{$|v|=1$ and $m > 0$},
$$
then $Q_v \equiv 0$.
\end{prop}

\begin{remark}
Below we will treat the case $|v|=1$ and $m=0$, which can be solved by an ODE argument in $d=1$ dimensions.
\end{remark}

\begin{remark}
The arguments above easily carry over to nonlocal nonlinearities in the equation stated in Lemma \ref{lem:virial}. For example, let $m \geq 0$ be given and suppose that  $Q_v \in H^{1/2}(\R^3)$ solves
$$
\sqrt{-\Delta+m^2} Q_v + i v \cdot \nabla Q_v + \omega Q_v - (|x|^{-1} \ast |Q_v|^2) Q_v = 0 \quad \mbox{in $\R^3$}
$$
with some $v \in \R^3$ and $\omega \in \R$. Then we have $Q_v \equiv 0$ whenever $|v| > 1$ holds. This shows that traveling solitary waves with speeds $|v| > 1$ do no exist for the {\em boson star equation} (see \cite{FL2}).
\end{remark}

\begin{proof}
Suppose (i) holds. Since we have
\be \label{ineq:riesz}
\left | (Q_v, i \Rcal_k Q_v) \right | \leq \| \Rcal_k Q_v\|_{L^2} \| Q_v \|_{L^2}\leq \|Q_v\|_{L^2}^2,
\ee
the identity in Lemma \ref{lem:virial} shows that  $|v| > 1$ implies that $Q_v \equiv 0$ must hold.

Let us now suppose that (ii) is true. By rotational symmetry, we can assume that $v=e_1$ without loss of generality. In this case, we must have equality in \eqref{ineq:riesz} for $k=1$, which implies that $\Rcal_1 Q_v = \lambda Q_v$ for some constant $\lambda$, i.\,e.,
$$
\frac{-i \xi_1}{\sqrt{|\xi|^2 + m^2}} \widehat{Q}_v(\xi) = \lambda \widehat{Q}_v(\xi) \quad \mbox{for almost every $\xi \in \R^d$.}
$$
For strictly positive $m> 0$, this is only possible if $\widehat{Q}_v = 0$ almost everywhere. [In other words, the operator $\Rcal_1$ does not have an eigenfunction in $L^2(\R^d)$ when $m>0$.]
\end{proof}

Next, we further investigate the case $|v|=1$ and vanishing mass parameter $m=0$ in Lemma \ref{lem:virial}. To do so, we introduce the following (mild) conditions on the nonlinearity appearing in the equation in Lemma \ref{lem:virial}.

\begin{assum} \label{assum:f}
The function $f : [0,\infty) \to \R$ has the following properties.
\begin{enumerate}
\item[(i)]  There exist constants $C \geq 0$ and $\alpha \geq 0$ such that $|f(t)| \leq C e^{\alpha t^2}$ for all $t \geq 0$.
\item[(ii)] The set of zeros $f^{-1}(\{ 0 \}) \subset [0, \infty)$ is discrete.
\end{enumerate}
\end{assum}

\begin{prop} \label{prop:no_Qv_2}
Let $Q_v \in H^{1/2}(\R^d)$ be as in Lemma \ref{lem:virial} above, where in addition we suppose that $f: [0,\infty) \to \R$  satisfies Assumption $\ref{assum:f}$. If $|v|=1$, $m=0$ and $d=1$, then $Q_v \equiv 0$.
\end{prop}

\begin{proof}
Recall that we must have equality in \eqref{ineq:riesz}, which in $d=1$ dimensions implies that $\Rcal_1 Q_v = \lambda Q_v$ for some constant $\lambda \in \C$, where it is easy to seen that $\lambda \in \{\pm i\}$. That is, the function $Q_v \in H^{1/2}(\R)$ has to be an eigenfunction of the Hilbert transform $\Rcal_1=H$ on the real line.  Thus one of the following cases is true.
\begin{itemize}
\item Case 1: $\mathrm{supp} \, \widehat{Q}_v \subset \R_{\geq 0}$ and $\lambda = -i$.
\item Case 2: $\mathrm{supp} \, \widehat{Q}_v \subset \R_{\leq 0}$ and $\lambda = +i$.
\end{itemize}
Suppose we are in Case 1. Then we have $\sqrt{-\DD} Q_v = -i \pt_x Q_v$. Hence $Q_v \in H^{1/2}(\R)$ is a weak solution of the ordinary differential equation
\be \label{eq:ODE}
-i \alpha \frac{d Q_v}{dx} + f(|Q_v|) Q_v = 0 \quad \mbox{in $\mathcal{D}'(\R)$},
\ee
with the real constant
$$
\alpha = 1+v \in \{0,2\},
$$
since, by assumption, we must  have $v \in \{ \pm 1 \}$.

If $\alpha = 0$, then \eqref{eq:ODE} implies that
$$
f(|Q_v|) Q_v = 0 \quad \mbox{almost everywhere}.
$$
Hence $Q_v \in D$ almost everywhere with the $D = f^{-1}( \{ 0 \} ) \cup \{ 0 \}$. Since $D$ must be a discrete set by Assumption \ref{assum:f} part (ii), it follows that $Q_v \equiv 0$ from Lemma \ref{lem:const} below.

Suppose now that $\alpha = 2$ holds.  We note that the growth condition on $f$ in Assumption \ref{assum:f} part (i) together with $Q_v \in H^{1/2}(\R)$ implies that $f(|Q|_v) Q_v \in L^1_{\mathrm{loc}}(\R)$ by a Moser--Trudinger type inequality; see Lemma \ref{lem:moser} below. Thus it follows from \eqref{eq:ODE} that $Q_v \in W^{1,1}_{\mathrm{loc}}(\R)$ and hence $(|Q_v|^2)' = Q_v' \overline{Q}_v + Q_v \overline{Q}_v'$ in the sense of weak derivatives. However, by using that $f$ is real-valued, we readily check from \eqref{eq:ODE} that
$$
\frac{d}{dx} (|Q_v|^2) = \frac{d Q_v}{dx} \overline{Q}_v + Q_v \frac{d \overline{Q}_v}{dx} = \frac{i}{2} f(|Q_v|)|Q_v|^2 - \frac{i}{2} f(|Q_v|) |Q_v|^2 \equiv 0 \quad \mbox{in $\mathcal{D}'(\R)$}.
$$
Thus we find that $x \mapsto |Q_v(x)|^2$ is a constant function on $\R$. But this implies that $Q_v \equiv 0$, since we have $Q_v \in L^2(\R)$.

By analogous arguments, we deduce that $Q_v \equiv 0$ also holds in Case 2.
\end{proof}

\section{Traveling Solitary Waves with Arbitrarily Small Energy}

\label{sec:no_scatt}

The goal of the present section is to proof of Theorem \ref{thm:no_scat}.

\subsection{Preliminaries}
We first collect some preliminary facts about traveling solitary wave whose profiles $Q_{v} \in H^{1/2}(\R^d)$ are obtained as optimizers for a suitable Gagliardo--Nirenberg--Sobolev (GNS) type interpolation inequality involving the velocity $v \in \R^d$ as a parameter.

Let $d \in \N$ be given and assume that $1 < p < p_*$. For any velocity $v \in \R^d$ with  $|v| < 1$, we define the quadratic form
\be \label{def:Tv}
T_v : \dot{H}^{1/2}(\R^d)\to \R,\quad u \mapsto T_v(u) = \int_{\R^d} (|\xi|- v \cdot \xi) |\widehat{u}(\xi)|^2 \, d \xi.
\ee
Clearly, we have the bounds
\be \label{ineq:Tv}
(1-|v|) \| u \|_{\dot{H}^{1/2}}^2 \leq T_v(u) \leq (1+|v|) \| u \|_{\dot{H}^{1/2}}^2,
\ee
which in particular show that $\sqrt{T_v(u)}$ defines a norm equivalent to $\| u \|_{\dot{H}^{1/2}}$.  Next, we recall that
\be
M(u) = \int_{\R^d} |u|^2 \, dx
\ee
denotes the squared $L^2$-norm of $u \in L^2(\R^d)$. We now have the following \emph{Gagliardo--Nirenberg--Sobolev type} interpolation inequality
\be \tag{GNS$_v$} \label{ineq:GNS_v}
\int_{\R^d} |u|^{p+1} \, dx \leq \Copt_{v,d,p} (T_v(u))^{\frac{d(p-1)}{2}} M(u)^{\frac{p+1}{2} - \frac{d(p-1)}{2}}
\ee
valid for all $u \in H^{1/2}(\R^d)$ with some constant $\Copt_{v,d,p} > 0$.

\begin{prop} \label{prop:GNS}
Let $d\in \N$ and $1 < p < p_*$. Then for any $v \in \R^d$ with $|v| < 1$, there exists an optimal constant $\Copt_{v,d,p} > 0$ for inequality \eqref{ineq:GNS_v}. Moreover, there exists an optimizer $Q_{v} \in H^{1/2}(\R^d) \setminus \{ 0 \}$ for \eqref{ineq:GNS_v} and it satisfies
$$
\hD \, Q_v + i v \cdot \nabla Q_v + Q_v -  |Q_v|^{p-1} Q_v = 0.
$$
\end{prop}

\begin{remark}
For zero velocity $v=0$, uniqueness of optimizers $Q_{0} \in H^{1/2}(\R^d)$ (modulo translation and phase) was shown in \cite{FrL,FrLS}. For the special case $d=1$ and $p=2$, we also have the explicit expression
$$
Q_0(x) = \frac{2}{1+x^2},
$$
which also arises as the (unique up to symmetries) solitary wave profile for the {\em Benjamin--Ono equation}; see \cite{AT}.
\end{remark}

\begin{remark}
For $0 <|v| < 1$, uniqueness of optimizers $Q_v \in H^{1/2}(\R^d)$ is an interesting open problem (except, of course, for small velocities $|v| \ll 1$ where perturbative arguments could be used.) For the special case $d=1$ and $p=3$, the limiting regime $|v| \to 1^-$ connects the optimizers $Q_v(\R)$ (in a suitable scaling) to solitons of the completely integrable {\em cubic Szeg\"o equation} on $\R$; see \cite{GLPR} for a detailed study in connection with turbulence phenomena.
\end{remark}

\begin{proof}[Proof of Proposition \ref{prop:GNS}]
The existence of optimizers can be deduced from standard variational methods applied to a suitable modification of the {\em ``Weinstein functional''} (see \cite{We}). In our case, the functional reads
\be
\mathcal{W}_{v,d,p}(u) = \frac{\displaystyle \int_{\R^d} |u|^{p+1}\, dx }{(T_v(u))^{\frac{d(p-1)}{2}} M(u)^{\frac{p+1}{2} - \frac{d(p-1)}{2}}}
\ee
defined for functions $u \in H^{1/2}(\R^d) \setminus \{ 0 \}$. The optimal constant for the corresponding interpolation estimate \eqref{ineq:GNS_v} is then given by the variational problem
\be
\Copt_{v,d,p} = \sup_{u \in H^{1/2}(\R^d) \setminus \{ 0 \}} \mathcal{W}_{v,d,p}(u) .
\ee
That this supremum is indeed  attained can be inferred from concentration-compactness methods; see, e.\,g., \cite{KLR} for $d=1$ and $p=3$. The general case can be handed in a similar fashion with straightforward modifications. Finally, it is elementary to check that any optimizer $Q_v \in H^{1/2}(\R^d)$ satisfies the equation in Proposition \ref{prop:GNS} -- after a suitable rescaling $Q_v(x) \mapsto a Q_v(bx)$ with some constants $a, b > 0$.
\end{proof}

For the rest of this section, let us write
$$
\Copt_v = \Copt_{v,d,p}
$$
for the optimal constants given by Proposition \ref{prop:GNS} above. Moreover, we shall use $X \lesssim Y$ to denote that $X \leq C Y$, where $C>0$ is a constant that only depends on $d$ and $p$. As usual, we write $X \sim Y$ to mean that both $X \lesssim Y$ and $Y \lesssim X$ hold.

\begin{lem} \label{lem:Cv_bound}
For $|v| < 1$, we have the estimate
$$
\Copt_{v} \sim (1-|v|)^{-\frac{d(p-1)}{2}}.
$$
\end{lem}

\begin{proof}
First, we show the universal upper bound
$$
\Copt_{v} \lesssim (1-|v|)^{-\frac{d(p-1)}{2}}.
$$
Indeed, recall that $T_v(u) \geq (1-v) \| |\nabla|^{1/2} u \|_{L^2}^2$. Thus if $\Copt_0>0$ denotes the optimal constant for \eqref{ineq:GNS_v} with $v=0$, we deduce that
\begin{align*}
\left ( \int_{\R^d} |u|^{p+1} \, dx \right ) & \leq \Copt_0 \left ( \int_{\R^d} | |\nabla|^{1/2} u |^2  \,dx \right )^{\frac{d(p-1)}{2}}  ( M(u) )^{\frac{p+1}{2}-\frac{d(p-1)}{2}} \\
& \leq \Copt_0 (1-|v|)^{-\frac{d(p-1)}{2}} ( T_v(u) )^{\frac{d(p-1)}{2}} (M(u) )^{\frac{p+1}{2}-\frac{d(p-1)}{2}},
\end{align*}
which proves the  universal upper bound $\Copt_v \leq \Copt_0(1-|v|)^{-\frac{d(p-1)}{2}}$.

To show the universal lower bound
$$
\Copt_v \gtrsim (1-|v|)^{-\frac{d(p-1)}{2}},
$$
we argue as follows. We remark that the one-dimensional case $d=1$ is particularly simple to deal with (and therefore the following discussion could be simplified in this case).

By a suitable (possibly improper) rotation in $\R^d$, we can henceforth assume that
$$
v = |v| e_1 = (|v|, 0, \ldots, 0) \in \R^d
$$
points in (positive) $x_1$-direction. Next, we fix a function $g \in H^{1/2}_+(\R)=\Pi_+ (H^{1/2}(\R)) = \{ f \in H^{1/2}(\R) : \mathrm{supp} \, \widehat{f} \subset \R_+ \}$, where we assume $g \not \equiv 0$. For $\eps > 0$ given and $d \geq 2$,  we let $\phi_\eps : \R^{d-1} \to \R$ be given by
$$
\phi_\eps(y) = c_\eps e^{-\eps |y|^2} \quad \mbox{for $y \in \R^{d-1}$},
$$
where $c_\eps > 0$ is chosen such that the normalization $\int_{\R^{d-1}} |\phi_\eps|^2 \, dx  = 1$ holds. For $d=1$, we just set $\phi_\eps \equiv 1$ in what follows.

Now, let $u_\eps \in H^{1/2}(\R^d)$ be given by its Fourier transform with
$$
\widehat{u}_\eps(\xi) = g(\xi_1) \phi_\eps(\xi_2, \ldots, \xi_d).
$$
Since the functions $\{ |\phi_\eps(\cdot)|^2\}_{\eps > 0}$ form an approximate identity on $\R^{d-1}$, we find
\begin{align*}
\| \widehat{u}_\eps \|_{\dot{H}^{1/2}(\R^d)}^2 & = \int_{\R^d} \sqrt{\xi_1^2 + \ldots + \xi_d^2} \, |\widehat{g}(\xi)|^2 |\phi_\eps(\xi_2, \ldots, \xi_d)|^2 \, d\xi \\
& \to \int_{\R} |\xi_1| |\widehat{g}(\xi_1)|^2 \, d \xi_1 = \| g \|_{\dot{H}^{1/2}(\R)}^2 \quad \mbox{as} \quad  \eps \to 0.
\end{align*}

Next, by using that $|v| < 1$, we find the elementary chain of inequalities for nonnegative numbers $\alpha, \beta \geq 0$ given by
$$
\sqrt{\alpha^2 + \beta^2} - |v| \alpha \leq \alpha + \beta - |v| \alpha \leq (1-|v|) \alpha + \beta \leq (1-|v|) \sqrt{\alpha^2+\beta^2} + \beta.
$$
Since $\mathrm{supp} \, \widehat{g} \subset [0,\infty)$ by assumption and hence $(|\xi|-|v| \xi_1) |\widehat{u}_\eps(\xi)|^2 = (|\xi|-|v| |\xi_1|) |\widehat{u}_\eps(\xi)|^2$, we can apply the elementary inequalities from above to $|\alpha| = |\xi_1|$ and $\beta = \sqrt{\xi_2^2 + \ldots + \xi_d^2}$. This yields
\begin{align*}
T_v(u_\eps) & = \int_{[0,\infty) \times \R^{d-1}} \sqrt{\xi_1^2+ \ldots + \xi_d^2} \, |\widehat{g}(\xi)|^2 |\phi_\eps(\xi_2, \ldots, \xi_d)|^2 \, d\xi_1 \, d \xi_2 \ldots d \xi_d \\
& \leq (1-|v|) \int_{\R^d} \sqrt{\xi_1^2 + \ldots + \xi_d^2} \, |\widehat{u}_\eps(\xi)|^2 \, d\xi \\
&  \quad + \| \widehat{g} \|_{L^2(\R)}^2 \int_{\R^{d-1}} \sqrt{\xi_2^2 + \ldots + \xi_d^2} \, |\phi_\eps(\xi_2, \ldots, \xi_d)|^2 \, d \xi_2 \ldots d \xi_d \\
& = (1-|v| ) \| u_\eps\|_{\dot{H}^{1/2}}^2 +  o(1) \quad \mbox{with $o(1) \to 0$ as $\eps \to 0$.}
\end{align*}
Since $\|  u_\eps \|_{\dot{H}^{1/2}(\R^d)} \to \|  g \|_{\dot{H}^{1/2}(\R)} \neq 0$ as $\eps \to 0$, we can thus find some sufficiently small $\eps_0 > 0$ such that
$$
T_v(u_{\eps_0}) \leq 2 (1-|v|) \|  u_{\eps_0} \|_{\dot{H}^{1/2}}^2.
$$
Therefore we obtain
$$
\Copt_v \gtrsim (1-|v|)^{-\frac{d(p-1)}{2}} \frac{ \| u_{\eps_0} \|_{L^p}^{p+1} }{ \| |\nabla|^{1/2} u_{\eps_0} \|_{L^2}^{d(p-1)} \| u_{\eps_0} \|_{L^2}^{(p+1)-d(p-1)}}  \gtrsim K (1-|v|)^{-\frac{d(p-1)}{2}},
$$
with some constant $K=K(u_{\eps_0})>0$. This completes the proof of Lemma \ref{lem:Cv_bound}.
\end{proof}

As a consequence of Lemma \ref{lem:Cv_bound} we obtain the following estimates.

\begin{lem} \label{lem:Qv_bounds}
For $|v| < 1$ and $Q_v \in H^{1/2}(\R^d)$  as in Proposition \ref{prop:GNS}, we have the bounds
$$
\| Q_v \|_{L^2}^2  \sim (1-|v|)^{d}  \quad  \mbox{and} \quad \| Q_v \|_{\dot{H}^{1/2}}^2 \lesssim (1-|v|)^{d-1}.
$$
\end{lem}

\begin{proof}
This follows from Pohozeav identities combined with the estimate for the optimal $\Copt_v$ in Lemma \ref{lem:Cv_bound}.  Indeed, from the identities in Lemma \ref{lem:poho} we conclude
\be
T_v(Q_v) \sim M(Q_v) \sim \int_{\R^d} |Q_v|^{p+1} \, dx.
\ee
Since $Q_v$ turns \eqref{ineq:GNS_v} into an equality, a short calculation yields
$$
M(Q_v)  \sim  \Copt_v^{-\frac{2}{p-1}} \sim (1-|v|)^{d},
$$
where the last step follows from Lemma \ref{lem:Cv_bound}. From the general bound \eqref{ineq:Tv} and $T_v(Q_v) \sim M(Q_v)$ we conclude that $\| Q_v \|_{\dot{H}^{1/2}}^2 \lesssim (1-|v|)^{d-1}$.
\end{proof}

Next, we show that non-trivial traveling solitary waves with speed $|v| < 1$ cannot scatter to a free half-wave.

\begin{prop} \label{prop:no_scat}
Let $d \in \N$ and $1 < p < p_*$. Suppose that $u_v(t,x) = e^{i t \omega} Q_{v, \omega}(x-vt)$, where $Q_v \in H^{1/2}(\R^d)$ solves \eqref{eq:Qv} with some $|v|  < 1$ and $\omega \geq 0$.

Then if there exists $\phi \in H^{1/2}(\R^d)$ with
$$
\lim_{t \to +\infty} \| u_v(t,x) - e^{-it \hD} \phi \|_{H^{1/2}} = 0,
$$
it holds that $\phi \equiv u_v(t) \equiv 0$. An analogous statement holds for negative times $t \to -\infty$.
\end{prop}

\begin{proof} We treat the cases $d=1$ and $d \geq 2$ separately as follows.

\subsubsection*{Case $d=1$}  Let $Q_v^\pm = \Pi^\pm Q_v$ and $\phi^\pm = \Pi^{\pm} \phi$ be the projections of $Q_v \in H^{1/2}(\R)$ and $\phi\in H^{1/2}(\R)$ onto the positive and negative Fourier frequencies, respectively. Likewise, we have $u_v^\pm (t) = \Pi^\pm u_v(t) = e^{it \omega} Q_v^\pm (x-vt)$.

Since $H^{1/2}(\R) = \Pi^+(H^{1/2}(\R)) \oplus \Pi^- (H^{1/2}(\R))$ is an orthogonal decomposition,
\begin{align*}
\| u_v(t) - e^{-it \hD} \phi \|_{H^{1/2}}^2 & = \| u_v^+(t) - e^{-it \hD} \phi^+ \|_{H^{1/2}}^2 +  \| u_v^-(t) - e^{-it \hD} \phi^- \|_{H^{1/2}}^2 \\
& =: A(t) + B(t),
\end{align*}
where $A(t) \to 0$ and $B(t) \to 0$ as $t \to +\infty$ by assumption. Now, by using that $e^{-it \hD}$ acts unitarily on the subspaces $\Pi^{\pm} (H^{1/2}(\R))$, we observe that
$$
A(t) = \| e^{it \hD} u_+(t) - \phi^+ \|_{H^{1/2}}^2 = \| e^{it \omega} Q_v^+(\cdot - (v-1) t) - \phi^+ \|_{H^{1/2}}^2
$$
where we used that $(e^{it \hD} f)(x) = (e^{t \pt_x}f)(x) = f(x +t)$ for all $f \in \Pi^+ (H^{1/2}(\R))$. From the fact that $A(t) \to 0$ as $t \to +\infty$ we deduce  that
$$
\mbox{$e^{it \omega} Q_v^+(\cdot - \delta t) \weakto \phi^+$ weakly in $L^2(\R)$ as $t \to +\infty$,}
$$
where we set $\delta = v-1$. But since it holds $|v| < 1$, we have $\delta \neq 0$ and it is elementary to verify that $e^{it \omega} f(\cdot - \delta t) \weakto 0$ weakly in $L^2(\R)$ as $t \to \infty$ for arbitrary $f \in L^2(\R)$. Thus, we conclude $\phi^+ \equiv 0$. An analogous argument shows that $\phi^- \equiv 0$, by using now that $\delta=v+1 \neq 0$ holds.

In summary, we obtain $\phi = \phi^+ + \phi^- \equiv 0$, which  implies that $\lim_{t \to + \infty} \| u_v(t) \|_{H^{1/2}}=0$. But since $\| u_v(t) \|_{H^{1/2}} = \|Q_v \|_{H^{1/2}}$ for all $t$, we conclude that $u_v(t) \equiv Q_v \equiv 0$ as claimed.

\subsubsection*{Case $d \geq 2$} In space dimensions $d \geq 2$, we can exploit dispersive estimates for $e^{-it \hD}$. First, we note that the hypothesis is equivalent to
$$
\lim_{t \to +\infty} \| e^{it \hD} u_v(t) - \phi \|_{H^{1/2}} = 0
$$
because $e^{-it \hD}$ is unitary on $H^{1/2}(\R^d)$.  Next, as in the previous step above, we now claim that
\be  \label{conv:weak0}
\mbox{$e^{i \omega t} e^{it \hD} Q_v(\cdot-vt) \weakto 0$ weakly in $L^2(\R^d)$ as $t \to +\infty$},
\ee
which would imply that $\phi \equiv 0$ and hence $u_v(t) \equiv Q_v \equiv 0$ in the same way as in the case $d=1$ above.

To prove \eqref{conv:weak0}, we first note that the $L^2$-norm
$$
\| e^{i\omega t} e^{i t  \hD} Q_v(x-vt) \|_{L^2} = \| Q_v \|_{L^2}
$$
is independent of $t$ (and, in particular, it is uniformly bounded in $t$). Hence, by density, it suffices to show that
\be  \label{eq:conv}
\langle e^{i \omega t} e^{it \hD} Q_v(\cdot-vt), \psi \rangle \to 0 \quad \mbox{as} \quad t \to +\infty
\ee
for any test function $\psi \in C^\infty_c(\R^d)$. To establish \eqref{eq:conv}, we first claim that
\be \label{eq:horman}
Q_v \in L^{(p+1)/p}(\R^d).
\ee
Indeed, we note that the map
$$
G=(\hD + i v \cdot \nabla + \omega)^{-1}
$$
with $|v| < 1$ and $\omega \geq 0$ is bounded on $L^q(\R^d)$ for any $1 < q < \infty$ by the Mikhlin--H\"ormander mutliplier theorem. Since we have $Q_v \in L^{p+1}(\R^d)$,  we deduce that $Q_v = G ( |Q_v|^{p-1} Q_v)$ belongs to $L^{(p+1)/p}(\R^d)$. [Alternatively, we could establish \eqref{eq:horman} using the decay estimate for $Q_v$ in Lemma \ref{lem:Qv_regdecay}.]

Now, with \eqref{eq:horman} available, we can  estimate as follows:
\begin{align*}
\left | \langle e^{i \omega t} e^{it \hD} Q_v(\cdot-vt), \psi \rangle  \right | & = \left | \langle e^{i \omega t} Q_v(\cdot-vt), e^{-it \hD} \psi \rangle \right | \\
& \leq \| Q_v \|_{L^{(p+1)/p}} \| e^{-it \hD} \psi \|_{L^{p+1}} \\
& \lesssim  \|Q_v \|_{L^{(p+1)/p}} t^{-(\frac{d-1}{2}) (2-\frac{1}{p+1})} \| (\hD)^{\frac{d+1}{2}} \psi \|_{L^{(p+1)/p}},
\end{align*}
 where in the last step we used a standard dispersive estimate for the half-wave propagator $e^{-it \hD}$. Thus we have found that
$$
\langle e^{i \omega t} e^{it \hD} Q_v(\cdot-vt), \psi \rangle = O \big (t^{-(\frac{d-1}{2}) ( 2-\frac{1}{p+1} )} \big )= o(1) \to 0 \quad \mbox{as} \quad t \to +\infty,
$$
which proves \eqref{eq:conv} and hence \eqref{conv:weak0} follows.

This completes the proof of Proposition \ref{prop:no_scat}, where we notice that the case of negative times $t \to -\infty$ follows in exactly the same fashion.
\end{proof}

\subsection{Proof of Theorem \ref{thm:no_scat}}
Suppose that $d \in \N$ and $1 < p <p_*$. Let  $\eps > 0$ be given.

For space dimensions $d \geq 2$, we can simply choose $Q_v \in H^{1/2}(\R^d) \setminus \{ 0 \}$ as given by Proposition \ref{prop:GNS}. Since $d \geq 2$, we can find $|v| < 1$ such that
$$
\| Q_v \|_{H^{1/2}}^2 = \| Q_v \|_{L^2}^2 + \| Q_v \|_{\dot{H}^{1/2}}^2 \lesssim (1-|v|)^{d} + (1-|v|)^{d-1} < \eps^2,
$$
using the bounds from Lemma \ref{lem:Qv_bounds}. By applying Proposition \ref{prop:no_scat}, we see that the traveling solitary waves $u(t,x) = e^{it}Q_v(x-vt)$ satisfy the conclusion of Theorem \ref{thm:no_scat}.

Next, we turn to the case of $d=1$ space dimension. Let $Q_v \in H^{1/2}(\R) \setminus \{ 0 \}$ be again given from Proposition \ref{prop:GNS}. By Lemma \ref{lem:Qv_bounds}, we have the bounds $\| Q_v \|_{L^2}^2 \lesssim (1-|v|)$ and $\| Q_v \|_{\dot{H}^{1/2}}^2 \lesssim 1$. For the parameter $\lambda > 0$, we define the rescaled function
$$
Q_v^{(\lambda)}(x) = \lambda^{\frac{1}{p-1}} Q_v(\lambda x),
$$
which is easily seen to solve
$$
  \hD  Q_v^{(\lambda)} + i v \cdot \nabla Q_v^{(\lambda)} + \lambda Q_v^{(\lambda)} -  |Q_v^{(\lambda)}|^{p-1} Q_v^{(\lambda)} = 0.
$$
Next, we notice that
$$
\| Q_v^{(\lambda)} \|_{L^2}^2 = \lambda^{\frac{2}{p-1}-1} \| Q_v \|_{L^2}^2 \lesssim \lambda^{\frac{2}{p-1}-1} (1-|v|), \quad \| Q_v^{(\lambda)} \|_{\dot{H}^{1/2}}^2 = \lambda^{\frac{2}{p-1}} \| Q_v \|_{\dot{H}^{1/2}}^2 \lesssim \lambda^{\frac{2}{p-1}}.
$$
With the choice of $\lambda = (1-|v|)^{\frac{p-1}{2p}}$, we thus obtain
$$
\| Q_v^{(\lambda)} \|_{H^{1/2}}^2 = \| Q_v^{(\lambda)} \|_{L^2}^2 + \| Q_v^{(\lambda)} \|_{\dot{H}^{1/2}}^2 \lesssim (1-|v|)^{\frac{p+3}{2p}} + (1-|v|)^{\frac{1}{p}} < \eps^2,
$$
provided we take $|v| < 1$ sufficiently close to 1. By applying Proposition \ref{prop:no_scat}, we see that $u(t,x) = e^{it \lambda} Q_v^{(\lambda)} (x-vt)$  provides the desired non-scattering solutions in Theorem \ref{thm:no_scat} in dimension $d=1$.

The proof of Theorem \ref{thm:no_scat} is now complete. \hfill $\qed$
\begin{appendix}

\section{Some Technical Results}

\begin{lem} \label{lem:const}
Let $D \subset \R$ be a discrete set. Suppose that $f \in H^{1/2}(\R^d)$ satisfies
$$
\mbox{$|f(x)| \in D$ for almost every $x\in \R^d$.}
$$
Then $f \equiv 0$.
\end{lem}

\begin{proof}
Since $f \in H^{1/2}(\R^d)$ implies $|f| \in H^{1/2}(\R^d)$, we can assume $f \geq 0$ without loss of generality. Furthermore, by changing $f$ on a set of zero measure if necessary, we can suppose that $f(x) \in D$ for all $x \in \R^d$.

For $a \in D$, let us denote $E_a = \{ x \in \R^d : f(x) = a \}$. Now, take $a \in D$ such that $|E_a| > 0$ holds (where $|\cdot|$ denotes the $d$-dimensional Lebesgue measure). We claim that
\be \label{eq:Ea}
f(x) = a \quad \mbox{for almost every $x \in \R^d$}.
\ee
Indeed, once this is shown, we conclude from $f \in L^2(\R^d)$ that $a=0$, which yields that $f \equiv 0$ as desired.

Thus it remains to show \eqref{eq:Ea} and let us write $E = E_a$ for notational convenience. Suppose that \eqref{eq:Ea} was false. Then $|E^c| > 0$ holds for  the complement $E^c = \R^d \setminus E$. Next, we recall that
\begin{align*}
\| f \|_{\dot{H}^{1/2}(\R^d)}^2 & = c_d \int \! \! \int_{\R^d \times \R^d} \frac{|f(x)-f(y)|^2}{|x-y|^{d+1}} \, dx \, dy
\end{align*}
with some constant $c_d > 0$. Since $D \subset \R$ is discrete, there is some $\delta > 0$ such that
$$
|f(x) - f(y)| \geq \delta \quad \mbox{for all $(x,y) \in E \times E^c$.}
$$
Therefore, we can estimate
$$
\int \! \! \int_{\R^d \times \R^d} \frac{|f(x)-f(y)|^2}{|x-y|^{d+1}} \, dx \, dy \geq  \delta^2 \int \! \! \int_{E \times E^c} \frac{dx \, dy}{|x-y|^{d+1}} = +\infty,
$$
where the last step follows from \cite[Corollary 2]{Brezis} using that both $|E| > 0$ and $|E^c| > 0$. Hence we find $f \not \in H^{1/2}(\R^d)$, which is absurd. Thus \eqref{eq:Ea} holds and the proof is complete.
\end{proof}

\begin{lem} \label{lem:moser}
For any $f \in H^{1/2}(\R)$ and $\alpha \geq 0$, it holds that $e^{\alpha |f|^2} f \in L^1_{\mathrm{loc}}(\R)$.
\end{lem}

\begin{proof}
This result is probably known to experts in the area of embeddings of Orlicz spaces into Sobolev spaces. For the reader's convenience, we provide some details as follows.

Let $\Omega \subset \R$ be a compact set and suppose that $f \in H^{1/2}(\R)$. First, we claim that
\be \label{ineq:moser1}
\int_\Omega e^{\gamma |f|^2} \, dx < +\infty \quad \mbox{for any $\gamma > 0$.}
\ee
Indeed, if this holds true, we readily deduce that $e^{\alpha |f|^2} f \in L^1_{\mathrm{loc}}(\R)$  from H\"older's inequality and $f \in L^2(\R)$.

To show estimate \eqref{ineq:moser1}, we recall a result from \cite[Theorem 1.6]{LL}, which combined with a simple scaling argument shows that, for each $\beta> 0$, we have
$$
\sup_{\| u \|_{H^{1/2}(\R)}^2 \leq \beta_0/\beta} \int_{\Omega} e^{\beta |u|^2} \, dx \leq C |\Omega|,
$$
where $C > 0$ and $\beta_0 > 0$ are universal constants. Next, we observe the pointwise bound $|f|^2 \leq 2 |f-\varphi|^2  + 2 |\varphi|^2 \leq 2 |f-\varphi|^2 + 2 \|\varphi\|_{L^\infty}$ for any bounded function $\varphi \in L^\infty(\R)$. Thus for every $\gamma > 0$ we find the bound
$$
\int_{\Omega} e^{\gamma|f|^2} \, dx \leq e^{2 \gamma \| \varphi \|_{L^\infty}^2} \int_{\Omega} e^{2 \gamma |f- \varphi|^2} \, dx \leq C e^{2 \gamma \| \varphi \|_{L^\infty}^2} |\Omega| < +\infty,
$$
provided we choose $\varphi \in C^\infty_c(\R)$ such that $\| f - \varphi \|_{H^{1/2}(\R)}^2 \leq \beta_0/2\gamma$, which is possible by density of $C^\infty_c(\R)$ in $H^{1/2}(\R)$. This proves \eqref{ineq:moser1}.
\end{proof}

\begin{lem} \label{lem:poho}
For $Q_{v} \in H^{1/2}(\R^d)$ as in Proposition \ref{prop:GNS}, we have the identities
\be \label{eq:poho_app1}
(Q_{v}, \hD  Q_{v}) + (Q_{v}, i v \cdot \nabla Q_{v}) +  \| Q_{v} \|_{L^2}^2 - \| Q_{v} \|_{L^{p+1}}^{p+1} = 0,
\ee
\be \label{eq:poho_app2}
(Q_{v}, \hD Q_{v}) + (Q_{v}, i v \cdot \nabla Q_{v}) - d \left ( \frac{p-1}{p+1}  \right ) \| Q_{v} \|_{L^{p+1}}^{p+1} = 0.
\ee
\end{lem}

\begin{proof}
Identity \eqref{eq:poho_app1} simply follows from integrating the equation for $Q_{v}$ against $\overline{Q}_{v}$. However, the rigorous proof of \eqref{eq:poho_app2} requires some care. By the estimates in Lemma \ref{lem:Qv_regdecay} below, we are allowed to integrate the equation against $\Lambda \overline{Q}_v$ with $\Lambda = x \cdot \nabla + d/2$. Elementary calculations then show
$$
(\Lambda Q_v, \hD Q_v) = (Q_v, \hD Q_v), \quad (\Lambda Q_v, i v \cdot \nabla Q_v) = (Q_v, i v \cdot \nabla Q_v),
$$
$$
(\Lambda Q_v, Q_v) = 0, \quad (\Lambda Q_v, |Q_v|^{p-1} Q_v) = d \left ( \frac{p-1}{p+1} \right ) \| Q_v \|_{L^{p+1}}^{p+1}.
$$
This completes the proof of \eqref{eq:poho_app2}.
\end{proof}

\begin{lem} \label{lem:Qv_regdecay}
Let $d \in \N$, $1 < p < p_*$, and $v \in \R^d$ with $|v| < 1$. Suppose that $Q_v \in H^{1/2}(\R^d)$ solves
$$
\hD Q + i v \cdot \nabla Q_v + Q_v -|Q_v|^{p-1} Q_v = 0 .
$$
Then $Q_v \in H^1(\R^d) \cap C_0(\R^d)$ and we have the decay estimates
$$
|Q_v(x)| + |\nabla Q_v(x)| \leq  \frac{A}{|x|^{d+1}}.
$$
with some constant $A>0$. In particular, we have $x \cdot \nabla Q_v \in L^2(\R^d)$.
\end{lem}

\begin{proof} We divide the proof as follows.

\subsubsection*{Case $v=0$}
Let us start with the special case $v=0$. A way to prove pointwise decay of $Q_v$ can be obtained by using that the resolvent operator $R = (\hD + 1)^{-1}$ has an integral kernel $G$ such that
\begin{enumerate}
\item[(i)] $G \in L^p(\R)$ for $1 < p < \infty$ with $1-\frac{1}{p} < 1/d$.
\item[(ii)] $|G(x)| \leq C |x|^{-(d+1)}$ for $|x| \geq 1$.
\end{enumerate}
Then, by using (i) and iterating the equation for $Q = Q_{v=0} \in H^{1/2}(\R^d)$, one finds that $Q \in H^1(\R^d)\cap L^q(\R^d)$ for any $q \in [2, \infty]$. Furthermore, by writing the equation for $Q$ as $Q=G \ast (|Q|^{p-1}Q)$ it follows that $Q \in C_0(\R^d)$, i.\,e., it $Q$ is continuous and vanishes at infinity.

To show the pointwise decay estimate for $Q$, we can adapt known arguments using the decay estimates for $G$. For the reader's convenience, we provide a brief sketch here. If we let $V(x) = |Q(x)|^{p-1} Q(x)$, we can write
$$
Q(x) = \int_{\R^d} G(x-y) V(y) Q(y) \, dy.
$$
Now, by iterating this equation and using initially that $V$ and $Q$ are continuous functions vanishing at infinity, we can conclude by using property (ii) for $G$ that we have
$$
|Q(x)| \leq A |x|^{-(d+1)} \quad \mbox{for all $x \in \R^d$}
$$
with some constant $A> 0$. We refer to \cite{H} for details, where a general decay estimate of solitary waver for fractional NLS is worked out.

Next, by differentiating the equation for $Q$ with respect to $x$, we obtain
\be \label{eq:kernel}
L \pt_{x_k} Q = 0 \quad \mbox{for $k=1,\ldots,d$},
\ee
where $L$ is the linearized operator given by
$$
L = \hD + 1 - p |Q|^{p-1}.
$$
Using the bounds (i) and (ii), arguments in the same fashion as sketched above show that $\pt_k Q$ as a solution of \eqref{eq:kernel} are continuous and satisfy the pointwise decay estimate
$$
|\pt_k Q(x)| \leq A |x|^{-(d+1)}.
$$
This concludes our sketch for the case $v=0$.

\subsubsection*{Case $0 < |v| < 1$}
We can extend the properties (i) and (ii) to the kernel $G_v(x-y)$ for the resolvent operator $R_v = (\hD + i v \cdot \nabla + 1)^{-1}$ provided that $|v| < 1$ holds. Indeed, this can be achieved by the following arguments. Note that $G_v = \mathcal{F}^{-1}\{ (|\xi| - v \cdot \xi +1)^{-1} \}$ and observe that
$$
\frac{1}{|\xi| - v \cdot \xi +1} = \int_0^\infty e^{-t} e^{-t ( |\xi| -  v \cdot \xi)} \, dt.
$$
Following an idea in \cite{FL2},  we can use the explicit formula for $\mathcal{F}^{-1}(e^{-t |\xi|})$ and analytic continuation to find
\be
G_v(x)  = c_d \int_0^\infty e^{-t} \frac{t}{(t^2 + (x+itv)^2)^{\frac{d+1}{2}}} \, dt,
\ee
with some constant $c_d > 0$. Note here that the complex number
$$
w= t^2 + (z+itv)^2 = (1-|v|^2) t^2 +  |x|^2 + 2it v \cdot x
$$
satisfies $|\arg z| < \pi/2$ and $|w| \geq |\mathrm{Re} \, w| = (1-|v|^2) t^2 + |x|^2$ for $t > 0$, $x \in \R^d$ and $|v| < 1$.  Using the trivial bound $e^{-t} \leq 1$ for $t >0$, we obtain the rough estimate
\be \label{ineq:Gv_bound}
|G_v(x)| \lesssim \int_0^{\infty} \frac{t}{((1-|v|^2) t^2 + |x|^2)^{\frac{d+1}{2}}} \, dt \lesssim A_v  \begin{dcases*} {|x|^{-d+1}} & for $d \geq 2$, \\
(\log |x|) + 1 & for $d=1$, \end{dcases*}
\ee
with some constant $A_v > 0$ depending on $|v| < 1$ and $d$. Furthermore, we find the bound
\begin{align*}
|G_v(x)| & \lesssim \frac{1}{|x|^{d+1}} \int_0^\infty \frac{e^{-t} t}{((1-|v|^2) t^2/|x|^2 +1 )^{\frac{d+1}{2}} } \, dt \\
& \lesssim \frac{1}{|x|^{d+1}} \int_0^\infty e^{-t} t \, dt \lesssim \frac{1}{|x|^{d+1}} \quad \mbox{for $|x| \geq 1$}.
\end{align*}
This shows (ii) for $G_v$ and together with \eqref{ineq:Gv_bound} we deduce that (i) holds too.

Having now established (i) and (ii) for $G_v$ with $|v| < 1$, we can now proceed in the same fashion as in the case $v=0$ described above. We omit the details.
\end{proof}

\section{On the Energy-Critical Case}

We give the proof of Theorem \ref{thm:HW_crit} showing that small data scattering fails for the energy-critical half-wave equation \eqref{eq:HW_crit} in general.

Let $d \geq 2$ and recall that $p_*+1 = \frac{2d}{d-1}$. Let us take $v \in \R^d$ with $|v| < 1$. We consider the variational problem
\be \label{def:sup_sob}
\Copt_{v,d,p_*} = \sup_{u \in \dot{H}^{1/2}(\R^d) \setminus \{ 0 \}} \frac{\displaystyle \int_{\R^d} |u|^{\frac{2d}{d-1}} \,dx  }{(T_v(u))^{\frac{d}{d-1}}},
\ee
where the functional $T_v : \dot{H}^{1/2}(\R^d) \to \R$ was introduced in \eqref{def:Tv} above. By known variational methods (e.\,g.~concentration-compactness methods) and using that $|v| < 1$, we deduce that the supremum in \eqref{def:sup_sob} is attained. Let us denote the maximizers by $W_v \in \dot{H}^{1/2}(\R^d) \setminus \{0 \}$ in what follows. After a suitable rescaling $W_v(x) \mapsto a W_v(bx)$ with constants $a,b > 0$, we can assume that $W_v \in \dot{H}^{1/2}(\R^d)$ satisfies the corresponding Euler-Lagrange equation
\be
\hD W_v + i v \cdot \nabla W_v - |W_v|^{\frac{2}{d-1}} W_v = 0 \quad \mbox{in $\R^d$}.
\ee
By integrating this equation against $\overline{W}_v \in \dot{H}^{1/2}(\R^d)$ and using its maximizing property for \eqref{def:sup_sob}, we directly find that
\be
T_v(W_v) = \int_{\R^d} |W_v|^{\frac{2d}{d-1}} \, dx \quad \mbox{and} \quad \Copt_{v,d,p_*} = (T_v(W_v))^{- \frac{1}{d-1}}.
\ee
On the other hand, by adapting the argument used in the proof of Lemma \ref{lem:Cv_bound}, we deduce
\be
\Copt_{v,d,p_*} \sim (1-|v|)^{-\frac{d}{d-1}}.
\ee
From this fact together with the general bound \eqref{ineq:Tv} we thus deduce
\be
\| W_v \|_{\dot{H}^{1/2}}^2 \leq (1-|v|)^{-1} \cdot T_v(W_v)  \lesssim (1-|v|)^{d-1}.
\ee
Since $d \geq 2$, we conclude that
\be
\|W_v\|_{\dot{H}^{1/2}} \to 0 \quad \mbox{as} \quad |v| \to 1^-.
\ee

Therefore, we can construct traveling solitary waves $u(t,x) = e^{it} W_v(x-vt) \in \dot{H}^{1/2}(\R^d)$ for the energy-critical half-wave equation \eqref{eq:HW_crit} with arbitrarily small energy norm by choosing $|v| < 1$ sufficiently close to 1. A straightforward adaption of the proof of Proposition \ref{prop:no_scat} now shows that there is no function $\phi \in \dot{H}^{1/2}(\R^d)$ such that
\be
\lim_{t \to +\infty} \| e^{it} W_v(x-vt) - e^{-it \hD} \phi \|_{\dot{H}^{1/2}} = 0.
\ee
This completes the proof of Theorem \ref{thm:HW_crit}. \hfill $\qed$
\end{appendix}

\end{document}